\documentclass[12pt]{amsart}
\usepackage{amssymb}
\usepackage{amsfonts}
\usepackage{latexsym}
\usepackage{amscd,enumerate}

\usepackage[all]{xy}

\input xy
\xyoption{all}

\numberwithin{equation}{section}

\newcommand{\Prim}{\text{{\rm Prim}}}
\newcommand{\Spec}{\text{{\rm Spec}}}

\newtheorem{lemma}{Lemma}[section]
\newtheorem{corollary}[lemma]{Corollary}
\newtheorem{theorem}[lemma]{Theorem}
\newtheorem{proposition}[lemma]{Proposition}
\newtheorem{remark}[lemma]{Remark}

\newtheorem{definitions}[lemma]{Definitions}

\newtheorem{examples}[lemma]{Examples}

\newtheorem*{notation}{Notation}

\newcommand{\Ker}{{\rm Ker}}
\newcommand{\Mod}{{\rm Mod}}
\newcommand{\Ann}{{\rm Ann}}

\begin{document}

\subjclass[2000]{Primary 16D70} \keywords{Leavitt path algebra, prime ideal, maximal tail, primitive ring}
\date{\today}
\title[Prime spectrum and primitive Leavitt path algebras]{Prime spectrum and primitive Leavitt path algebras}
\author{G. Aranda Pino}
\address{Departamento de \'Algebra, Universidad Complutense de Madrid,  28040 Madrid, Spain.}
\email{gonzaloa@mat.ucm.es}\author{E. Pardo}
\address{Departamento de Matem\'aticas, Universidad de C\'adiz,
Apartado 40, 11510 Puerto Real (C\'adiz),
Spain.}\email{enrique.pardo@uca.es}
\urladdr{http://www2.uca.es/dept/matematicas/PPersonales/PardoEspino/index.HTML}
\author{M. Siles Molina}
\address{Departamento de \'Algebra, Geometr\'{\i}a y Topolog\'{\i}a,
Universidad de M\'alaga,  29071 M\'alaga, Spain.}\email{mercedes@agt.cie.uma.es}
\thanks{The first author was supported by a Centre de Recerca Matem\`{a}tica Fellowship within the Research Programme
``Discrete and Continuous Methods on Ring Theory". The second author was partially supported by the DGI and European Regional Development Fund,
jointly, through Project MTM2004-00149, by PAI III projects FQM-298 and P06-FQM-1889 of the Junta de Andaluc\'{\i}a, and by the Comissionat per
Universitats i Recerca de la Generalitat de Catalunya. The first and third authors by the Spanish MEC and Fondos FEDER jointly through projects
MTM2004-06580-C02-02 and MTM2007-60333, and PAI III projects FQM-336 and FQM-1215 of the Junta de Andaluc\'{\i}a; and all three authors by the Consolider Ingenio
``Mathematica" project CSD2006-32 by the MEC}

\begin{abstract} In this paper a bijection between the set of prime ideals of a Leavitt path algebra $L_K(E)$ and
a certain set which involves maximal tails in $E$ and the prime spectrum of $K[x,x^{-1}]$ is established. Necessary and sufficient conditions on the
graph $E$ so that the Leavitt path algebra $L_K(E)$ is primitive are also found.
\end{abstract}

\maketitle


\section*{introduction}

Leavitt path algebras of row-finite graphs have been recently introduced in \cite{AA1} and \cite{AMP}. They have become a subject of
significant interest, both for algebraists and for analysts working in C*-algebras. The Cuntz-Krieger algebras $C^*(E)$ (the C*-algebra counterpart
of these Leavitt path algebras) are described in \cite{R}. The algebraic and analytic theories, while sharing some striking similarities, they present some remarkable differences, as was shown for instance in the ``Workshop on Graph Algebras'' held at the University of M\'alaga (see \cite{Work}), and more
deeply in the subsequent enlightening work of Tomforde \cite{T}.

For a field $K$, the algebras $L_K(E)$ are natural generalizations of the algebras investigated by Leavitt in \cite{Le}, and are a specific type of
path $K$-algebras associated to a graph $E$ (modulo certain relations). The family of algebras which can be realized as the Leavitt path algebras of
a graph includes matrix rings ${\mathbb M}_n(K)$ for $n\in \mathbb{N}\cup \{\infty\}$ (where ${\mathbb M}_\infty(K)$ denotes matrices of countable
size with only a finite number of nonzero entries), the Toeplitz algebra, the Laurent polynomial ring $K[x,x^{-1}]$, and the classical Leavitt
algebras $L(1,n)$ for $n\ge 2$. Constructions such as direct sums, direct limits, and matrices over the previous examples can be also realized in
this setting. But, in addition to the fact that these structures indeed contain many well-known algebras, one of the main interests in their study is
the comfortable pictorial representations that their corresponding graphs provide.

A great deal of effort has been focused on trying to unveil the algebraic structure of $L_K(E)$ via the graph nature of $E$.
Concretely, the literature on Leavitt path algebras includes necessary and sufficient conditions on a graph $E$ so that the corresponding Leavitt
path algebra $L_K(E)$ is simple \cite{AA1}, purely infinite simple \cite{AA2}, exchange \cite{APS}, finite dimensional \cite{AAS1}, locally finite
(equivalently noetherian) \cite{AAS2} and semisimple \cite{AAPS}. Another remarkable approach has been the research (performed quite intensively in
\cite{AMP}, and only slightly in \cite{AAPS}) of their monoids of finitely generated projective modules $V(L_K(E))$.

The aim of this paper is to determine the prime and primitive Leavitt path algebras, which has a twofold motivation. First, from the purely algebraic
point of view, this enterprise is a compulsory as well as a natural one. Throughout the mathematical literature, knowing the prime and primitive
spectra of rings (also of associative, Lie and Jordan algebras, etc) has been crucial in order to succeed to give structural theorems (or in order to
simply gain a better understanding of the given algebraic system). Classically, one of the uses of the prime spectrum for commutative rings is to
carry information over from Algebra to Topology and vice versa via the so-called Zariski topology (several generalizations of this construction for
noncommutative rings have been achieved \cite{ZTW, KT}). As for the primitive ideals of a ring, they naturally correspond to the irreducible
representations of it, which in turn represent unquestionable tools in their analysis. Therefore, the knowledge of the prime and primitive Leavitt
path algebras can be regarded as a fundamental and necessary step towards the ultimate goal of the classification of these algebras. In addition, the
prime and primitive questions are natural ones in the following sense: it is known (see \cite[Proposition 6.1]{AA3} or \cite[Proposition 1.1]{AMMS})
that every Leavitt path algebra is semiprime, and recently it has been proved that every Leavitt path algebra is also semiprimitive \cite[Proposition
6.3]{AA3}. These results obviously raised the questions of whether or not every Leavitt path algebra is also prime or primitive.

The second motivation springs out of the complete description of the primitive spectrum of a graph C*-algebra $C^*(E)$ carried out by Hong and
Szyma\'nski in \cite{HS}. Concretely, in \cite[Corollary 2.12]{HS}, the authors found a bijection between the set $\Prim(C^*(E))$ of primitive ideals
of $C^*(E)$ and some sets involving maximal tails and points of the torus ${\mathbb T}$. This result parallels one of the main result of this article
(Theorem \ref{correspondenciatotal}). However, there is one subtlety here: it is known that every primitive C*-algebra is prime and the converse
holds for separable C*-algebras \cite{D}. It turns out that every graph C*-algebra is separable and therefore the concepts of primeness and
primitivity are indistinguishable for $C^*(E)$. This is no longer the case for Leavitt path algebras $L_K(E)$, and in fact Theorem
\ref{correspondenciatotal} deals with the prime spectrum of a Leavitt path algebra whereas its analytic counterpart \cite[Corollary 2.12]{HS}
considers primitive ideals.

Hence, the primitive case for $L_K(E)$ deserves a different examination to the prime case, and Theorem \ref{primitivos} states the primitive
characterization for Leavitt path algebras. This result does not correspond verbatim to the characterization of primitive (equivalently prime) graph
C*-algebras, the difference being the possibility of having cycles without exits. This difference in graph criteria of a certain property for
$L_K(E)$ and $C^*(E)$ is not new, as it too showed up in the computation of the stable rank for $L_K(E)$ in \cite[Theorem 7.6]{APS}, and of the
stable rank for $C^*(E)$ in \cite[Theorem 3.4]{DHSz}.

The article is organized as follows. The Preliminaries section includes the basic definitions and examples that will be used throughout. In addition,
we describe several graph constructions and more specific but general properties of $L_K(E)$ that will be of use in the rest of the paper.

In Section 2 the first step of the investigation of prime ideals is carried out. We start by analyzing some subset of vertices of the graph called
maximal tails and then show that they are in one-to-one correspondence with the set of graded prime ideals of $L_K(E)$. Further along in Section 2,
several lemmas concerning prime but not necessarily graded ideals are obtained. Those are key ingredients in the study of the prime spectrum in the
the following section. Informally, these results tell us how to uniquely obtain, out of a graded but not necessarily prime ideal $I$, two things: a
maximal tail and a graded prime ideal contained in $I$.

The classification of all prime ideals is accomplished in Section 3. Some preliminary results discussing ideals generated by $P_c(E)$ (that is, the
vertices for which there are cycles without exits based at them) are settled. Those and other partial results finally pave the way for the proof of
one of the main results of the paper (Theorem \ref{correspondenciatotal}), which exhibits a bijection between the set of prime ideals of $L_K(E)$,
and the set formed by the disjoint union of the maximal tails of the graph ${\mathcal M}(E)$ and the cartesian product of maximal tails for which
every cycle has an exit ${\mathcal M}_\tau(E)$ and the nonzero prime ideals of the Laurent polynomial ring $\Spec(K[x,x^{-1}]^*)$. As noted before,
Theorem \ref{correspondenciatotal} is the algebraic analog of the graph C*-algebra result stated in \cite[Corollary 2.12]{HS}. However, it is worth
mentioning that their proofs are certainly unrelated since they involve totally different methods and what is more, neither can be (at least readily)
obtained from the other.

The natural subsequent step is taken in Section 4, where the primitive Leavitt path algebras are determined. In order to achieve this goal, several
results on simple right $L_K(E)$-modules are established. Then, in the other main theorem of this paper (Theorem \ref{primitivos}), necessary and
sufficient conditions are given so that a Leavitt path algebra $L_K(E)$ is left (equivalently right) primitive. In contrast with the prime spectrum
correspondence, this characterization of primitive Leavitt path algebras lacks a graph C*-algebra version.

\section{Preliminaries}

A \emph{(directed) graph} $E=(E^0,E^1,r,s)$ consists of two countable sets $E^0,E^1$ and maps $r,s:E^1 \to
E^0$. The elements of $E^0$ are called \emph{vertices} and the elements of $E^1$ \emph{edges}. If $s^{-1}(v)$
is a  finite set for every $v\in E^0$, then the graph is called \emph{row-finite}.  Throughout this paper we
will be concerned only with row-finite graphs. If $E^0$ is finite, then, by the row-finite hypothesis, $E^1$
must necessarily be finite as well; in this case we say simply that $E$ is \emph{finite}.

A vertex which emits no edges is called a \emph{sink}. A \emph{path} $\mu$ in a graph $E$ is a sequence of
edges $\mu=e_1\dots e_n$ such that $r(e_i)=s(e_{i+1})$ for $i=1,\dots,n-1$. In this case, $s(\mu):=s(e_1)$ is
the \emph{source} of $\mu$, $r(\mu):=r(e_n)$ is the \emph{range} of $\mu$, and $n$ is the \emph{length} of
$\mu$. For $n\ge 2$ we define $E^n$ to be the set of paths of length $n$, and $E^*=\bigcup_{n\ge 0} E^n$ the
set of all paths. Throughout the paper $K$ will denote an arbitrary field.

\smallskip

Let $K$ be a field  and $E$ a directed graph. Denote by $KE$ the $K$-vector space which has as a basis the
set of paths. It is possible to define an algebra structure on $KE$ as follows: for any two paths
$\mu=e_1\dots e_m, \nu=f_1\dots f_n$, we define $\mu\nu$ as zero if $r(\mu)\neq s(\nu)$ and as $e_1\dots e_m
f_1\dots f_n$ otherwise. This $K$-algebra is called the \textit{path algebra} of $E$ over $K$.

\smallskip

We define the {\em Leavitt path $K$-algebra} $L_K(E)$, or simply $L(E)$ if the base field is understood, as
the $K$-algebra generated by a set $\{v\mid v\in E^0\}$ of pairwise orthogonal idempotents, together with a
set of variables $\{e,e^*\mid e\in E^1\}$, which satisfy the following relations:

(1) $s(e)e=er(e)=e$ for all $e\in E^1$.

(2) $r(e)e^*=e^*s(e)=e^*$ for all $e\in E^1$.

(3) $e^*e'=\delta _{e,e'}r(e)$ for all $e,e'\in E^1$.

(4) $v=\sum _{\{ e\in E^1\mid s(e)=v \}}ee^*$ for every $v\in E^0$ that emits edges.

\smallskip

Relations (3) and (4) are called of Cuntz-Krieger.

The elements of $E^1$ are called \emph{real edges}, while for $e\in E^1$ we call $e^\ast$ a \emph{ghost edge}. The set $\{e^*\mid e\in E^1\}$ will be
denoted by $(E^1)^*$.  We let $r(e^*)$ denote $s(e)$, and we let $s(e^*)$ denote $r(e)$.  If $\mu = e_1 \dots e_n$ is a path, then we denote by
$\mu^*$ the element $e_n^* \dots e_1^*$ of $L(E)$, and by $\mu^0$ the set of its vertices, i.e., $\{s(\mu_1),r(\mu_i)\mid i=1,\dots,n\}$. It was
shown in \cite[Lemma 1.5]{AA1} that every monomial in $L(E)$ is of the form:  $kv$, with $k\in K$ and $v\in E^0$, or $k e_1\dots e_m f_1^*\dots
f_n^*$ for $k\in K$, $m, n\in \mathbb{N}$, $e_i,f_j\in E^1$. For any subset $H$ of $E^0$, we will denote by $I(H)$ the ideal of $L(E)$ generated by
$H$.

Note that if $E$ is a finite graph then we have $\sum _{v\in E^0} v=1_{L(E)}$.  On the other hand, if $E^0$ is infinite, then by \cite[Lemma
1.6]{AA1} $L(E)$ is a nonunital ring with a set of local units.  In fact, in this situation, $L(E)$ is a ring with {\it enough  idempotents} (see
e.g. \cite{F} or \cite{T}), and we have the decomposition $L(E)=\oplus_{v\in E^0}L(E)v$ as left $L(E)$-modules. (Equivalently, we have
$L(E)=\oplus_{v\in E^0}vL(E)$ as right $L(E)$-modules.)

\begin{examples} {\rm By considering some basic configurations one can realize many algebras as the Leavitt path algebra of some graph. Thus,
for instance, the ring of Laurent polynomials $K[x,x^{-1}]$ is the Leavitt path algebra of the graph

$$
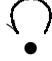
\xymatrix{{\bullet} \ar@(ur,ul)  }$$

Matrix algebras $M_n(K)$ can be achieved by considering a line graph with $n$ vertices and $n-1$ edges

$$\xymatrix{{\bullet} \ar [r]  & {\bullet} \ar [r]  & {\bullet} \ar@{.}[r] & {\bullet} \ar [r]  & {\bullet} }$$

Classical Leavitt algebras $L(1,n)$ for $n\geq 2$ are obtained as $L(R_n)$ where $R_n$ is the rose with $n$ petals graph

$$
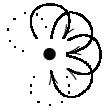
\xymatrix{{\bullet} \ar@(ur,dr)  \ar@(u,r)  \ar@(ul,ur)  \ar@{.} @(l,u) \ar@{.} @(dr,dl) \ar@(r,d) &  }$$ \smallskip

Of course, combinations of the previous examples are possible. For example, the Leavitt path algebra of the graph

$$
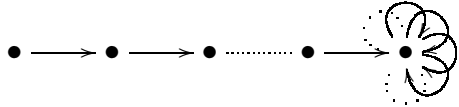
\xymatrix{{\bullet} \ar [r]  & {\bullet} \ar [r]  & {\bullet}
\ar@{.}[r] & {\bullet} \ar [r]  & {\bullet}
 \ar@(ur,dr)  \ar@(u,r)  \ar@(ul,ur)  \ar@{.} @(l,u) \ar@{.} @(dr,dl) \ar@(r,d) & }$$ \smallskip

is $M_n(L(1,m))$, where $n$ denotes the number of vertices in the graph and $m$ denotes the number of loops. In addition, the algebraic counterpart
of the Toeplitz algebra $T$ is the Leavitt path algebra of the graph $E$ having one loop and one exit
$$
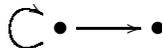
\xymatrix{{\bullet} \ar@(dl,ul) \ar[r] & {\bullet}  }$$
}
\end{examples}

\medskip

It is shown in \cite{AA1} that $L(E)$ is a ${\mathbb Z}$-graded $K$-algebra, spanned as a $K$-vector space by $\{pq^* \mid p,q$ are paths in $E\}$.
In particular, for each $n\in \mathbb{Z}$, the degree $n$ component $L(E)_n$ is spanned by elements of the form  $pq^*$ where $l(p)-l(q)=n$. The
degree of an element $x$, denoted $deg(x)$, is the lowest number $n$ for which $x\in \bigoplus_{m\leq n} L(E)_m$.

For us, by a {\it countable} set we mean a set which is either finite or countably infinite. The symbol ${\mathbb M}_\infty(K)$ will denote the
$K$-algebra of matrices over $K$ of countable size but with only a finite number of nonzero entries.

We will analyze the structure of various graphs in the sequel. An important role is played by the following three concepts. An edge $e$ is an {\it
exit} for a path $\mu = e_1 \dots e_n$ if there exists $i$  such that $s(e)=s(e_i)$ and $e\neq e_i$.  If $\mu$ is a path in $E$, and if
$v=s(\mu)=r(\mu)$, then $\mu$ is called a \emph{closed path based at $v$}. If $s(\mu)=r(\mu)$ and $s(e_i)\neq s(e_j)$ for every $i\neq j$, then $\mu$
is called a \emph{cycle}. A graph which contains no cycles is called \emph{acyclic}.

An edge $e$ is an {\it exit} for a path $\mu = e_1 \dots e_n$ if there exists $i$ such that $s(e)=s(e_i)$ and $e\neq e_i$. We say that a graph $E$
satisfies \emph{Condition} (L) if every cycle in $E$ has an exit.

We define a relation $\ge$ on $E^0$ by setting $v\ge w$ if there is a path $\mu\in E^*$ with $s(\mu)=v$ and $r(\mu)=w$.  A subset $H$ of $E^0$ is
called \emph{hereditary} if $v\ge w$ and $v\in H$ imply $w\in H$. A hereditary set is \emph{saturated} if every vertex which feeds into $H$ and only
into $H$ is again in $H$, that is, if $s^{-1}(v)\neq \emptyset$ and $r(s^{-1}(v))\subseteq H$ imply $v\in H$. Denote by $\mathcal{H}_E$ the set of
hereditary saturated subsets of $E^0$.

The set $T(v)=\{w\in E^0\mid v\ge w\}$ is the \emph{tree} of $v$, and it is the smallest hereditary subset of $E^0$  containing $v$. We extend this
definition for an arbitrary set $X\subseteq E^0$ by $T(X)=\bigcup_{x\in X} T(x)$. The \emph{hereditary saturated closure} of a set $X$ is defined as
the smallest hereditary and saturated subset of $E^0$ containing $X$. It is shown in \cite{AMP, BHRS} that the hereditary saturated closure of a set
$X$ is $\overline{X}=\bigcup_{n=0}^\infty \Lambda_n(X)$, where
\begin{enumerate}

\item[] $\Lambda_0(X)=T(X)$, and
\item[] $\Lambda_n(X)=\{y\in E^0\mid
s^{-1}(y)\neq \emptyset$ and $r(s^{-1}(y))\subseteq \Lambda_{n-1}(X)\}\cup \Lambda_{n-1}(X)$, for $n\ge 1$.
\end{enumerate}

Recall that an ideal $J$ of $L(E)$ is graded if and only if it is generated by idempotents; in fact, $J=I(H)$, where $H=J\cap E^0\in \mathcal{H}_E$.
(See the proofs of \cite[Proposition 4.2 and Theorem 4.3]{AMP}.) We will use this fact freely throughout.

We recall here some graph-theoretic constructions which will be of interest. For a hereditary subset of $E^0$, the \emph{quotient graph} $E/H$ is
defined as
$$(E^0\setminus H, \{e\in E^1|\ r(e)\not\in H\}, r|_{(E/H)^1}, s|_{(E/H)^1}),$$ and the \emph{restriction graph} is
$$E_H=(H, \{e\in E^1|\ s(e)\in H\}, r|_{(E_H)^1}, s|_{(E_H)^1}).$$

Sometimes it is useful to view $L(E)$ constructed as the quotient of the path algebra of a certain graph as follows: recall that given a graph $E$
the  \textit{extended graph} of $E$ is defined as the new graph $\widehat{E}=(E^0,E^1\cup (E^1)^*,r',s')$ where $(E^1)^*=\{e_i^*:e_i\in E^1\}$ and
the functions $r'$ and $s'$ are defined as $$r'|_{E^1}=r,\ s'|_{E^1}=s,\ r'(e_i^*)=s(e_i) \hbox{ and } s'(e_i^*)=r(e_i).$$

For a field $K$ and a row-finite graph $E$, the Leavitt path algebra of $E$ with coefficients in $K$ can also be regarded as the path algebra over
the extended graph $\widehat{E}$, with relations:

\begin{enumerate}
\item[(CK1)] $e_i^*e_j=\delta_{ij}s(e_j)$ for every $e_j\in E^1$ and $e_i^*\in (E^1)^*$.

\item[(CK2)] $v_i=\sum_{\{e_j\in E^1:r(e_j)=v_i\}}e_je_j^*$ for every $v_i\in E^0$ which is not a source.
\end{enumerate}

Thus, an element of $L(E)$ will be of the form $\overline x$, with $x\in K\widehat{E}$. In fact, by \cite[Lemma 1.5]{AA1}, $x$ can be chosen as a
linear combination of vertices and elements of the form $pq^\ast$, with $p, q\in E^\ast$.

This alternative description of $L(E)$ allows us to define, for ${\overline x} \in L(E)$, the following
$$\mathcal{R}_{\overline x}=\left\{\sum p_iq_i^\ast \in K\widehat{E} \quad \vert \quad \overline{x} =
\overline{\sum p_iq_i^\ast}\right\}.$$ Consider an element $a=e_1\dots e_rf_1^\ast\dots f_s^\ast\in K\widehat{E}$, with $e_i, f_j \in E^1$. We say
that $s$ is the \textit{degree} of $e_1\dots e_rf_1^\ast\dots f_s^\ast$ \textit{in ghost edges}, and denote it by $\textrm{degge}(a)$. If $a\in KE$,
then we say that $a$ has zero degree in ghost edges, while the \textit{degree in ghost edges} of $f_1^\ast\dots f_s^\ast$ is $s$. For $a\in
K\widehat{E}$, $a=\sum_i p_iq_i$, with $p_i, q_i^\ast\in E^\ast$, the \textit{degree} of $a$ \textit{in ghost edges} is:
$max\{\textrm{degge}(p_iq_i^\ast)\}$. Finally, the \textit{degree in ghost edges} of an element $\overline x$ of the Leavitt path algebra $L(E)$ is
defined by:
$$\textrm{degge}(\overline{x}):= min \{\textrm{degge}(y)\quad \vert \quad y\in \mathcal{R}_{\overline x}\}.$$

\section{Prime Ideals and Maximal Tails}

The main goal of this section is the study maximal tails and their relation with prime (graded or not) ideals of $L(E)$. These connections will be
essential in the prime spectrum correspondence results (Theorem \ref{correspondenciatotal}).

Let us recall first the definition of maximal tail (which is a particular case of that of \cite{BHRS}): for a graph $E$, a nonempty subset
$M\subseteq E^0$ is said to be a \emph{maximal tail} if it satisfies the following properties:

\begin{enumerate}

\item[(MT1)] If $v\in E^0$, $w\in M$ and $v\geq w$, then $v\in M$.
\item[(MT2)] If $v\in M$ with $s^{-1}(v)\ne \emptyset$, then there exists $e\in E^1$ with $s(e)=v$ and $r(e)\in M$.
\item[(MT3)] For every $v,w\in M$ there exists $y\in M$ such that $v\geq y$ and $w\geq y$.
\end{enumerate}

\smallskip

\begin{lemma}\label{maximaltailhersat}
Let $E$ be a graph. Then, $M\subseteq E^0$ satisfies Conditions {\rm (MT1)} and {\rm (MT2)} if and only if $H=E^0\setminus M\in {\mathcal H}_E$.
\end{lemma}
\begin{proof}
Suppose first that $M$ is a maximal tail. Consider $v\in H$ and $w\in E^0$ such that $v\geq w$. If $w\not\in H$ then $w\in M$, and by Condition (MT1)
we get $v\in M=E^0\setminus H$, a contradiction. This shows that $H$ is hereditary. Now, let $v\in E^0$ with $s^{-1}(v)\neq \emptyset$, and suppose
that $r(s^{-1}(v))\subseteq H$. If $v\not\in H$ then by Condition (MT2), there exists $e\in s^{-1}(v)$ such that $r(e)\not\in H$, a contradiction.
This proves that $H$ is saturated.

Let us see the converse. Take $v\in E^0$ and $w\in M$ such that $v\geq w$. If $v\not\in M$ then, as $H$ is hereditary, we get that $w\in H$. Consider
now $v\in M$ with $s^{-1}(v)\neq \emptyset$. If for every $e\in s^{-1}(v)$ we have that $r(e)\not\in M$, then that means $r(s^{-1}(v))\subseteq H$,
and by saturation we obtain $v\in H$, a contradiction.
\end{proof}

\begin{notation} {\rm Following \cite{BHRS}, given $X\subseteq E^0$ we denote $$\Omega(X)=\{w\in E^0\setminus X\ |\ w\not\geq v \text{ for every }v\in X
\}.$$}
\end{notation}

\begin{lemma}\label{mtone} If $M\subseteq E^0$ satisfies Condition {\rm (MT1)}, then $\Omega(M)=E^0\setminus M$.
\end{lemma}
\begin{proof}
By definition $\Omega(M)\subseteq E^0\setminus M$. Now, let $w\in E^0\setminus M$. If $v\in M$ and $w\geq v$, by Condition (MT1) we get that $w\in
M$, a contradiction, so $E^0\setminus M\subseteq \Omega(M)$, as desired.
\end{proof}

\begin{corollary}\label{maximaltailomega} Let $E$ be a graph. If $M\subseteq E^0$ satisfies Conditions {\rm (MT1)} and {\rm (MT2)}, then
$\Omega(M)\in {\mathcal H}_E$.
\end{corollary}
\begin{proof}
Apply Lemmas \ref{maximaltailhersat} and \ref{mtone}.
\end{proof}

Recall that a graded ideal $I$ of a graded ring $R$ is said to be
\emph{graded prime} if for every pair of graded ideals $J, K$ of $R$
such that $JK\subseteq I$, it is necessary that either $J\subseteq
I$ or $K\subseteq I$. The definition of prime ideal is analogous to
the previous one by eliminating the condition of being graded. It
follows by \cite[Proposition II.1.4]{NvO} that for an algebra graded
by an ordered group (as it is the case of Leavitt path algebras), a
graded ideal is graded prime if and only if it is prime.

It will be useful to recall that in \cite[Remark 5.5]{APS} it was shown that if $J,K\in \mathcal{H}_E$, then $I(J)I(K)=I(J\cap K)$. We will use this
fact without referencing it.

For the sake of completion, we re-state here the following proposition:

\begin{proposition}\label{gradedprime} {\rm (\cite[Proposition 5.6]{APS})}
Let $E$ be a graph, and let $H\in \mathcal{H}_E$. Then, the following are equivalent:
\begin{enumerate}
\item The ideal $I(H)$ is (graded) prime.
\item $M=E^0\setminus H$ is a maximal tail.
\end{enumerate}
\end{proposition}

The following definitions can be found in \cite{BHRS}.

\begin{definitions} {\rm Let $M$ be a subset of $E$. A path in $M$ is a path $\alpha$ in $E$ with $\alpha^0\subseteq M$. We say that a path $\alpha$ in
$M$ has \emph{an exit in $M$} if there exits $e\in E^1$ an exit for $\alpha$ such that $r(e)\in M$. For a graph $E$, we denote by ${\mathcal M}(E)$
the set of maximal tails of $E$. We denote by ${\mathcal M}_\gamma (E)$ the set of maximal tails $M$ such that every closed simple path $p$ in $M$
has an exit in $M$. We will also denote ${\mathcal M}_\tau(E)={\mathcal M}(E)\setminus{\mathcal M}_\gamma(E)$.}
\end{definitions}

The following notation will be useful throughout the sequel.

\begin{notation} {\rm Keeping in mind that gauge-invariant ideals in graph C*-algebras correspond to graded ideals in Leavitt path algebras, we can
adapt some notation of \cite{HS} to our situation. Concretely, given a ${\mathbb Z}$-graded algebra $A$, we will denote by $\Spec_\gamma(A)$ the set
of all prime ideals of $A$ which are graded, and by $\Spec_\tau(A)$ the set of all prime ideals $L(E)$ which are not graded. Then
$\Spec(A)=\Spec_\gamma(A)\cup\Spec_\tau(A)$. As usual, we denote by $\Spec(A)^*$ the set $\Spec(A)\setminus\{0\}$.}
\end{notation}

\begin{lemma}\label{Mgamma} Let $E$ be a graph. Let $I$ be an ideal of $L(E)$. Let $H=I\cap E^0$ and $M=E^0\setminus H$.
If $M\in {\mathcal M}_\gamma(E)$ then $I=I(H)$.
\end{lemma}
\begin{proof}
First we suppose that $H$ is nonempty. By \cite[Lemma 3.9]{AA1} $H\in {\mathcal H}_E$, and by \cite[Lemma 2.3]{APS} $L(E)/I(H)\cong L(E/H)$. Clearly
$I(H)\subseteq I$. Suppose that $I(H)\neq I$, then $0\neq I/I(H)\triangleleft L(E/H)$. Note that, as $M\in {\mathcal M}_\gamma(E)$ by hypothesis,
then $E/H$ satisfies Condition (L). Thus, we are in a position to apply the same reasoning in \cite[Proposition 3.3]{APS} to reach a contradiction.

In the case when $H=\emptyset$, the condition $M\in {\mathcal
M}_\gamma(E)$ is just having Condition (L) in the graph $E$. Then,
if $I\ne 0$, an application of \cite[Proposition 6]{AA2} yields
$H\neq \emptyset$, a contradiction.
\end{proof}

\begin{lemma}\label{nongradedgivesmaximaltailtau} Let $E$ be a graph. Let $I$ be a non-graded prime ideal of $L(E)$.
Let $H=I\cap E^0$, then:
\begin{enumerate}[{\rm (i)}]
\item $I(H)\triangleleft L(E)$ is (graded) prime.
\item $M=E^0\setminus H\in {\mathcal M}_\tau(E)$.
\end{enumerate}
\end{lemma}
\begin{proof}
(i). By \cite[Lemma 3.9]{AA1} we know that $H\in {\mathcal H}_E$. Now, consider graded ideals $I_1,I_2$ of $L(E)$ such that $I_1I_2\subseteq I(H)$.
Find $H_i\in {\mathcal H}_E$ with $I_i=I(H_i)$, for $i=1,2$. As $I(H)\subseteq I$ and $I$ is prime, we have that $I(H_i)\subseteq I$, for some $i$.
Then, for this $i$ we get $H_i\subseteq I(H_i)\cap E^0\subseteq I\cap E^0=H$, so that $I(H_i)\subseteq I(H)$, as we wanted.

\medskip

(ii). Apply (i) and Proposition \ref{gradedprime} to get that $M$ is a maximal tail. If $M\in {\mathcal M}_\gamma(E)$, then Lemma \ref{Mgamma} gives
that $I=I(H)$, contradicting the fact that $I$ is not graded.
\end{proof}

We end this section by providing algebraic characterizations of Condition (L) and Conditions (L) plus (MT3), that will appear in the sequel. First we
need the following definitions, which are particular cases of those appearing in \cite[Definition 1.3]{DHSz}:

Let $E$ be a graph, and let $\emptyset \ne H\in \mathcal{H}_E$. Define
$$F_E(H)=\{ \alpha =(\alpha_1, \dots ,\alpha_n)\mid \alpha _i\in
E^1, s(\alpha _1)\in E^0\setminus H, r(\alpha _i)\in E^0\setminus H \mbox{ for } i<n, r(\alpha _n)\in H\}.$$ Denote by $\overline{F}_E(H)$ another
copy of $F_E(H)$. For $\alpha\in F_E(H)$, we write $\overline{\alpha}$ to denote a copy of $\alpha$ in $\overline{F}_E(H)$. Then, we define the graph
${}_HE=({}_HE^0, {}_HE^1, s', r')$ as follows:
\begin{enumerate}

\item ${}_HE^0=({}_HE)^0=H\cup F_E(H)$. \item ${}_HE^1=({}_HE)^1=\{ e\in
E^1\mid s(e)\in H\}\cup \overline{F}_E(H)$.
\item For every $e\in E^1 \mbox{ with } s(e)\in H$, $s'(e)=s(e)$ and $r'(e)=r(e)$. \item For every $\overline{\alpha}\in
\overline{F}_E(H)$,  $s'(\overline{\alpha})=\alpha$ and $r'(\overline{\alpha})=r(\alpha)$.
\end{enumerate}

\begin{proposition}\label{characterizations} Let $E$ be a graph.
\begin{enumerate}[{\rm (i)}]
\item $E$ satisfies Conditions {\rm (L)} and {\rm (MT3)} if and only if $I\cap J\cap E^0\neq \emptyset$ for
every nonzero ideals $I$ and $J$ of $L(E)$.
\item $E$ satisfies Condition {\rm (L)} if and only if $I\cap E^0\neq \emptyset$ for every nonzero ideal $I$ of $L(E)$.
\end{enumerate}
\end{proposition}
\begin{proof}
(i). Suppose that $E$ satisfies Conditions (L) and (MT3) and take nonzero ideals $I$ and $J$ of $L(E)$. Apply \cite[Proposition 6]{AA2} to find
vertices $v\in I$ and $w\in J$. Use Condition (MT3) to find $u\in E^0$ such that $v,w\geq u$ and paths $\mu, \nu$ such that $s(\mu)=v$, $s(\nu)=w$
and $r(\mu)=r(\nu)=u$. Thus $u=\mu^*v\mu=\nu^*w\nu \in I\cap J\cap E^0$.

Let us see the converse. By Proposition \ref{gradedprime}, $E$ satisfies Condition (MT3). Suppose now that there is a cycle without exists $c$ based
at $v$. Let $J$ denote the ideal of $L(E)$ generated by $v+c$. By a standard argument (see \cite[Proof of Theorem 3.11]{AA1}) $v\not\in J$. If $w\in
J$ for some $w\in E^0$, as we have Condition (MT3), there exists $u$ such that $v,w\geq u$. Because $c$ has no exits $u\in c^0$, so that for some
path $\tau$, we have $\tau=w\tau v$. This gives $v=\tau^*w\tau \in J$, a contradiction.

(ii). Apply \cite[Proposition 6]{AA2} to show that Condition (L) implies $I\cap E^0\neq \emptyset$ for every nonzero ideal $I$ of $L(E)$.

To see the converse, suppose that $c$ is a cycle without exits and write $H=\overline{c^0}$. By \cite[Lemma 1.2]{AP} $L(_HE)\cong I(H)$ via an
isomorphism $\Phi:L(_HE)\to I(H)$ such that for every $v\in I(H)$, $v=\Phi(v)$. Take a nonzero ideal $J$ of $L(_HE)$, then $\Phi(J)$ is an ideal of
$I(H)$. By \cite[Lemma 3.21]{T}, $\Phi(J)$ is an ideal of $L(E)$. Use the hypothesis to show that $\Phi(J)$ contains a vertex $w$ which is in $I(H)$,
hence $w\in J$ because $\Phi(w)=w$. This shows that the graph $_HE$ satisfies that every ideal of $L(_HE)$ contains a vertex. On the other hand, as
shown in \cite[Proposition 3.6 (iii)]{AAPS}, $_HE$ is a comet tail. Thus, it satisfies Condition (MT3). Now, consider the ideal $J'$ of $L(_HE)$
generated by $v+c$. We can prove as in (i) that $J'$ does not contain vertices, a contradiction.
\end{proof}

\begin{remark} {\rm The fact that Condition (L) implies $I\cap E^0\neq \emptyset$ for every nonzero ideal $I$ of $L(E)$ was first proved (although not
explicitly stated in this form) in \cite[Proposition 6]{AA2}. Despite its simplicity, this is a recurrently invoked fact in a great number of proofs
that have followed. What Proposition \ref{characterizations} (ii) shows then is that the converse of this well-known statement holds too. In
addition, Proposition \ref{characterizations} (i) provides a generalization of this aforementioned result, which in turn happens to be equivalent to
the left (or right) semiprimitivity of $L(E)$, as will be shown in Theorem \ref{primitivos}.}
\end{remark}

\section{The prime spectrum correspondence}

In this section the computation of the prime spectrum of the Leavitt path algebra is completed. The bijection between the set of prime ideals of
$L(E)$ and certain families of maximal tails together with the set of nonzero prime ideals of $K[x,x^{-1}]$ is fully achieved in Theorem
\ref{correspondenciatotal}.

First we will need some preliminary results that will be useful tools in both directions of the correspondence of that Theorem.

As in \cite{AAPS}, we denote by $P_c(E)$ the set of vertices in the cycles without exits of $E$.

\begin{lemma}\label{casoKE} Let $E$ be a graph and $J$ an ideal of $L(E)$ such that $J\cap E^0=\emptyset$.
Then $J\cap KE \cap L(E)u \subseteq I(P_c(E))$ for every $u\in E^0$.
\end{lemma}
\begin{proof}
We can assume $J\neq 0$. Apply \cite[Proposition 2.2]{S} to find $0\neq x=xu\in J\cap KE$. Write $x=\sum_{i=1}^r k_i \alpha_i$, with $0\neq k_i\in
K$, $\alpha_i=\alpha_iu\in E^*$ for every $i$ and $\alpha_i\neq \alpha_j$ for every $i\neq j$ and assume that $deg(\alpha_i)\leq deg(\alpha_{i+1})$
for every $i=1,\dots,r-1$. We will prove that $u\in I(P_c(E))$ by induction on the number $r$ of summands.

Note that $r\neq 1$ as otherwise we would have $k_1^{-1}\alpha_1^*x=u\in J$, a contradiction to the hypothesis. So the base case for the induction is
$r=2$. Suppose first that $deg(\alpha_1)=deg(\alpha_2)$. In this case, since $\alpha_1\neq \alpha_2$, we get $\alpha_1^*\alpha_2=0$ so that
$k_1^{-1}\alpha_1^*x=u\in J$, a contradiction again. This gives $deg(\alpha_1)<deg(\alpha_2)$ and then $\alpha_1^*x=k_1u+k_2e_1\dots e_t$ for some
$e_1,\dots,e_t\in E^1$. By multiplying on the left and right hand sides by $u$ we get
$$y_1:=u\alpha_1^*xu=k_1u+k_2ue_1\dots e_tu\in J.$$ Observe that $u$ and $e_1\dots e_n$ have different degrees
and since $k_1u\neq 0$ we obtain that $y_1\neq 0$. Moreover, as $J$ does not contain vertices we have that $c:=ue_1 \dots e_t u\neq 0$ is a closed
path based at $u$. We will prove that $c$ does not have exits: suppose on the contrary that there exist $w\in T(u)$ and $e,f\in E^1$ such that $e\neq
f$, $s(e)=s(f)=w$, $c=aweb=aeb$ for some $a,b\in E^*$. Then $\nu=af$ satisfies $\nu^*c=f^*a^*aeb=f^*eb=0$ so that $\nu^*y_1\nu=k_1r(\nu)\in J$, again
a contradiction. This is saying that $u\in P_c(E)$ so, in particular, $x=xu\in I(P_c(E))$.

Let us assume the result holds for $r$ and prove it for $r+1$. Assume then that $x=xu=\sum_{i=1}^{r+1}k_i \alpha_i$ and distinguish two situations.

First, consider $deg(\alpha_j)=deg(\alpha_{j+1})$ for some $j=1,\dots,r$. The element $\alpha_j^*xu\alpha_j=\alpha_j^*xu\alpha_ju\in J$ is nonzero as
follows: clearly each monomial remains with positive degree as $deg(\alpha_j^*\alpha_i\alpha_j)=deg(\alpha_i)\geq 0$. Moreover, at least
$\alpha_j=\alpha_j^*\alpha_j\alpha_j$ appears in the expression for $\alpha_j^*xu\alpha_j$ because if we had $\alpha_j=\alpha_j^*\alpha_i\alpha_j$
for some $i\neq j$, then $deg(\alpha_i)=deg(\alpha_j)$ which implies $\alpha_j^*\alpha_i=0$ and therefore $\alpha_i=0$, a contradiction. This shows
that $\alpha_j^*xu\alpha_j$ has at least a nonzero monomial, and because distinct elements of $KE$ are linearly independent (see \cite[Lemma
1.1]{S}), then $\alpha_j^*xu\alpha_j\neq 0$. Now, this element has at most $r$ summands because $\alpha_j^*\alpha_{j+1}\alpha_j=0$ and it satisfies
the induction hypothesis, so that $u\in P_c(E)$.

The second case is when $deg(\alpha_i)<deg(\alpha_{i+1})$ for every $i=1,\dots,r$. Then $0\neq \alpha_1^*x=k_1u+\sum_{i=2}^{r+1}k_i\beta_i$ with
$\beta_iu=\beta_i\in E^*$. Multiply again as follows:
$$y_2:=u\beta_{r+1}^*u\alpha_1^*xu\beta_{r+1}u=k_1u+\sum_{i=2}^{r+1}u\beta_{r+1}^*u\beta_iu\beta_{r+1}u\in J.$$
A similar argument to the previous paragraph shows that $y_2$ is nonzero so that, in case some monomial of $y_2$ becomes zero, then $y_2$ is
satisfies the induction hypothesis, therefore $u\in P_c(E)$. If this is not the case, since $\beta_{r+1}$ has maximum degree among the $\beta_i$,
then
$$y_2=k_1u+k_2\gamma_1+k_3\gamma_1\gamma_2+\dots +k_{r+1}\gamma_1\dots\gamma_r,$$ where $\gamma_i$ are closed
paths based at $u$. Let us focus on $\gamma_1$. By proceeding in a similar fashion as before, we can conclude that it cannot have exists as otherwise
there would exist a path $\delta$ with $s(\delta)=u$ and $\delta^*\gamma_1=0$. That would give $\delta^*y_2\delta=k_1r(\delta)\in J$, a
contradiction. Then, $\gamma_1$ is a cycle without exits so that $u\in P_c(E)$, and finally $x=xu\in I(P_c(E))$.
\end{proof}

\begin{proposition}\label{idealsinvertices} Let $E$ be a graph and $J$ an ideal of $L(E)$ such that $J\cap E^0=\emptyset$. Then $J \subseteq I(P_c(E))$.
\end{proposition}
\begin{proof}
Let $0\neq x\in J$, and write $x=\sum xu_i$ for some $u_i\in E^0$ with $0\neq xu_i$. As $J$ is an ideal, $0\neq xu_i\in J$, so that we can assume
without loss of generality that $0\neq x=xu$.

We will show, by induction on the degree in ghost edges, that if $xu\in J$, with $u\in E^0$, then  $xu\in I(P_c(E))$. If  $\textrm{degge}(xu)$, the
result follows by Lemma \ref{casoKE}. Suppose the result true for degree in ghost edges strictly less than $\textrm{degge}(xu)$ and show it for
$\textrm{degge}(xu)$.

Write $x=\sum_{i=1}^r\beta_ie_i^\ast+\beta$, with $\beta_i\in L(E)$, $\beta=\beta u\in KE$ and $e_i\in  E^1$, being $e_i\neq e_j$ for every $i\neq
j$. Then  $xue_i=\beta_i+\beta e_i\in J$; since $\textrm{degge}(xue_i)< \textrm{degge}(xu)$, by the  induction hypothesis $\beta_i+\beta e_i\in
I(P_c(E))$, for every $i\in \{1, \dots, r\}$.

If $u=\sum_{i=1}^re_ie_i^\ast$, then $xu=\sum_{i=1}^r\beta_i e_i^\ast +\sum_{i=1}^r\beta e_ie_i^\ast=$ $\sum_{i=1}^r(\beta_i+\beta e_i)e_i^\ast\in
I(P_c(E))$, and we have finished.

If $u=\sum_{i=1}^re_ie_i^\ast+\sum_{j=1}^sf_jf_j^\ast$ (where $f_j\in E^1$), then $xuf_j=\beta f_j\in J\cap  KE$. By Lemma \ref{casoKE} $\beta f_j\in
I(P_c(E))$, for every $j\in \{1, \dots, s\}$, hence  $xu=\sum_{i=1}^r(\beta_i+\beta e_i)e_i^\ast+ \sum_{j=1}^s\beta f_jf_j^\ast \in I(P_c(E))$.
\end{proof}

For a graph $E$, let $\{c_j\}_{j\in \Lambda}$ be the set of all different cycles without exits. By abusing of notation, identify two cycles that have
the same vertices. Then we can obtain the following

\begin{corollary}\label{idealprimo}
Let $J$ be a prime ideal of a Leavitt path algebra $L(E)$ which does not contain vertices. Then
$$\bigoplus_{i\in \Lambda^\prime}
{\mathbb M}_{n_i}(K[x, x^{-1}])\subseteq J \subseteq \bigoplus_{i\in \Lambda} {\mathbb M}_{n_i}(K[x, x^{-1}]),$$ where $\Lambda^\prime$ has exactly
one element less than $\Lambda$, $|\Lambda|\leq \aleph_0$ and $n_i\in {\mathbb N} \cup \{\infty\}$.
\end{corollary}
\begin{proof}
We will show $$I(\{c_j^0\}_{j\in\Lambda^\prime}) \subseteq J \subseteq I(P_c(E)).$$ Then, apply \cite[Proposition 3.6 (iii)]{AAPS}.

Suppose that there exist $z_1\in I(c_1)$ and $z_2\in I(c_2)$, for $c_1$ and $c_2$ different cycles without exits in $L(E)$ and such that $z_1,
z_2\notin J$. By \cite[Proposition 3.6 (i)]{AAPS} $\overline{z_1} \overline{I(P_c(E))}\overline{z_2}=\overline{z_1 I(P_c(E)) z_2}=0$. Since $J$ is a
prime ideal and $J\subseteq I(P_c(E))$, by Proposition \ref{idealsinvertices}, $\overline{I(P_c(E))}:=I(P_c(E))/J$ is a prime ring. This means
$\overline{z_1}=0$ or $\overline{z_2}=0$, that is $z_1\in J$ or $z_2\in J$, a contradiction. This shows our claim.
\end{proof}

\begin{corollary}\label{dentrodelideal} Let $F$ be a graph such that there is a unique cycle $\mu$ without
exits (but there might be other cycles with exits).
\begin{enumerate}[{\rm (i)}]
\item If $F\in {\mathcal M}(F)$ and $H\in {\mathcal H}_F\setminus \{\emptyset\}$, then $\overline{\mu^0}\subseteq H$.
\item If $J$ is an ideal of $L(F)$ such that $J\cap F^0=\emptyset$, then $J\subseteq I(\overline{\mu^0})$.
\end{enumerate}
\end{corollary}
\begin{proof}
(i). Applying \cite[Lemma 2.1]{HS}, we know that $\Omega(F^0)=\Omega(\mu^0)$, but since $\Omega(F)=\emptyset$, then this means that for every $w\in
F^0\setminus \mu^0$ we have $w\geq_F v$ for some $v\in \mu^0$. Now, given $h\in H$, as $\mu$ is a cycle we in fact have that $h\geq v$ for every
$v\in \mu^0$, and as $H$ is hereditary, this means that $\mu^0\subseteq H$. Now, because $H$ is also saturated we get $\overline{\mu^0}\subseteq H$.

\medskip

(ii) It is a particular case of Proposition \ref{idealsinvertices}.
\end{proof}

We recall here some definitions which were introduced in \cite{AAPS}. We say that an infinite path $\gamma=(e_n)_{n=1}^{\infty}$ \emph{ends in a
cycle} if there exists $m\geq 1$ and a cycle $c$ such  that the infinite subpath $(e_n)_{n=m}^{\infty}$ is just the infinite path $ccc\dots$. We say
that a graph $E$ is a \emph{comet} if it has exactly one cycle $c$, $T(v)\cap c^0\neq \emptyset$ for every vertex $v\in E^0$, and every infinite path
ends in the cycle $c$.

\medskip

Next propositions will be the pieces from which the main theorem of this section (Theorem \ref{correspondenciatotal}) will rely on.

\begin{proposition}\label{correspondencianograduadounlado} Let $E$ be a graph. There is a map $$\Theta:\Spec_\tau(L(E))\to {\mathcal M}_\tau(E)\times
\Spec(K[x,x^{-1}])^*.$$
\end{proposition}
\begin{proof}
Let $J$ be a prime ideal of $L(E)$ which is not graded. As the zero
ideal $\{0\}$ is graded, then $J\neq 0$. Consider $H=E^0\cap J\in
{\mathcal H}_E$ by \cite[Lemma 3.9]{AA1}. Then write $F=E/H$ so that
\cite[Lemma 2.1]{APS} gives that $L(F)\cong L(E)/I(H)$. Note that in
the case $H=\emptyset$ we simply have $F=E$ and we do not invoke any
result. Thus, Lemma \ref{nongradedgivesmaximaltailtau} gives that
$I(H)$ is graded prime and that $M=E^0\setminus H=F^0\in {\mathcal
M}_\tau(E)$. In particular this is saying that $L(F)$ is a prime
ring as $L(F)\cong L(E)/I(H)$.

Moreover, the ideal ${\mathcal J}=J/I(H)$ is prime in $L(F)$. To see this, first note that $I(H)\subseteq J$ but $I(H)\neq J$, as $J$ is nongraded by
hypothesis. Hence ${\mathcal J}\neq 0$. Furthermore, $$L(F)/{\mathcal J}\cong \frac{L(E)/I(H)}{J/I(H)}\cong L(E)/J$$ is a prime ring as $J$ is a
prime ideal in $L(E)$, so that ${\mathcal J}$ is a prime ideal in $L(F)$. Obviously it is not graded because otherwise it would imply that the ideal
$J$ to which it lifts is graded too.

Now, since $F^0\in {\mathcal M}_\tau(E)$, we will prove that $F^0\in {\mathcal M}_\tau(F)$. Clearly $F^0\setminus F^0=\emptyset$ is hereditary and
saturated in $F$, so that by Lemma \ref{maximaltailhersat} $F^0$ satisfies Conditions (MT1) and (MT2). Let us check Condition (MT3): take $v,w\in
F^0$. Since $F^0$ is a maximal tail in $E$, there exists $y\in F^0$ such that $v,w\geq_E y$, which means that there exist $p,q\in E^*$ such that
$s(p)=v, s(q)=w$ and $r(p)=r(q)=y$. Then, since $y\not\in H$, by hereditariness we have that $(p^0\cup q^0)\cap H=\emptyset$, and thus
$p^0,q^0\subseteq F^0$, which implies, by the way that $F$ is defined, that $v,w\geq_F y$. Finally, we can find a cycle $c$ in $F$ without exits in
$F$ when seen inside $E$, but this same cycle will not have exits in $F$ when regarded in $F$. This proves that $F^0\in {\mathcal M}_\tau(F)$.

Applying \cite[Lemma 2.1]{HS} to $F$ we get that there exists a unique cycle $\mu$ in $F$ without exits (but there could be other cycles with
exists). In this case we also have that $\emptyset=\Omega(F)=\Omega(\mu^0)$, or in other words, every vertex in $F^0$ connects to the cycle $\mu$.

Note that since $J\cap E^0=H$, then ${\mathcal J}\cap (E/H)^0={\mathcal J}\cap F^0=\emptyset$, so that we are in position to apply Corollary
\ref{dentrodelideal}(ii) to get that ${\mathcal J}\subseteq I(\overline{\mu^0})$. Now, by \cite[Lemma 2.1]{AP} we obtain $I(\overline{\mu^0})\cong
L(_{\overline{\mu^0}}F)$ as nonunital rings. In the notation of \cite{AAPS}, we have $P_\mu(F)=\mu^0$ so that $\overline{\mu^0}=\overline{P_\mu(F)}$.
First, we can show that every infinite path in $F$ ends in the cycle $\mu$ by just readapting the ideas in \cite[Proposition 3.6 (iii)]{AAPS}.
Moreover, this fact also implies that $\mu$ is the only cycle in $_{\overline{\mu^0}}F$, because any other cycle would produce an infinite path which
would not end in $\mu$. Clearly, by the way $F$ and $_{\overline{\mu^0}}F$ were constructed, every vertex in the latter connects to $\mu$.

This proves that $_{\overline{\mu^0}}F$ is in fact a comet, so that invoking \cite[Proposition 3.5]{AAPS} one gets that $L(_{\overline{\mu^0}}F)\cong
{\mathbb M}_n(K[x, x^{-1}])$, where $n\in \mathbb{N}$ if $_{\overline{\mu^0}}F$ is finite, or $n=\infty$ otherwise. By the composition of the two
previously determined isomorphism, we have a univocally defined $K$-algebra isomorphism $$\phi_\mu:I(\overline{\mu^0})\to {\mathbb M}_n(K[x,
x^{-1}]).$$

We will show now that ${\mathcal J}$ is a prime ideal in $I(\overline{\mu^0})$. Consider $A,B$ ideals of $I(\overline{\mu^0})$ such that $J\subseteq
A,B$ and $AB\subseteq J$. Since $I(\overline{\mu^0})$ is (isomorphic to) the Leavitt path algebra of $_{\overline{\mu^0}}F$, it has a set of local
units so that an application of \cite[Lemma 3.21]{T} yields that $A,B$ are ideals of $L(F)$ as well, but ${\mathcal J}$ was prime in $L(F)$ so that
$A\subseteq J$ or $B\subseteq J$, as we needed.

Then, $\phi_\mu({\mathcal J})$ is a prime ideal in ${\mathbb M}_n(K[x,x^{-1}])$, and it is well known that in this case there exists a unique ideal
$P$ of $K[x,x^{-1}]$ such that $\phi_\mu({\mathcal J})={\mathbb M}_n(P)$. Moreover, this ideal $P$ is prime in $K[x,x^{-1}]$ (see for instance
\cite{L}). Moreover, note that $P\neq 0$ because ${\mathcal J}\neq 0$.

That way we have associated a maximal tail $M\in {\mathcal M}_\tau(E)$ and a prime ideal $P$ in $K[x,x^{-1}]$ to $J$. In other words we have defined
$\Theta(J)=(M,P)$.
\end{proof}

\begin{proposition}\label{correspondencianograduadoelotrolado} Let $E$ be a graph. There is a map $$\Lambda:{\mathcal M}_\tau(E)\times
\Spec(K[x,x^{-1}])^* \to \Spec_\tau(L(E)).$$
\end{proposition}
\begin{proof}
Pick $P\neq 0$ any prime ideal in $K[x,x^{-1}]$ and $M\in {\mathcal M}_\tau(E)$. As $K[x,x^{-1}]$ is an Euclidean domain, we have that every nonzero
prime ideal in $K[x,x^{-1}]$ is maximal.

On the other hand, by \cite[Lemma 2.1]{HS}, there exists a cycle $\mu$ contained in $M$ but without exits in $M$. This cycle is unique (up to a
permutation of its edges) and $\Omega(M)=\Omega(\mu^0)$. Let $H=E^0\setminus M\in {\mathcal H}_E$ and $F=E/H$. Note that $F^0=M$, and that by the way
that $F$ is defined, $\mu^0\subseteq F^0$ and $\mu^1\subseteq F^1$. The fact that $\mu\subseteq E$ does not have exits in $M$ translates to the fact
that $\mu$ does not have exits when seen inside the graph $F$. The same reasoning used in Proposition \ref{correspondencianograduadounlado} shows
that $I(\overline{\mu^0})\cong L(_{\overline{\mu^0}}F)\cong {\mathbb M}_m(K[x, x^{-1}])$ for some $m\in \mathbb{N}\cup\{\infty\}$. As in the proof of
Proposition \ref{correspondencianograduadounlado}, we can consider the $K$-algebra isomorphism $\phi_\mu:I(\overline{\mu^0})\to {\mathbb M}_n(K[x,
x^{-1}])$.

Clearly ${\mathbb M}_m(P)$ is a maximal ideal \cite{L} in ${\mathbb M}_m(K[x, x^{-1}])$ so that ${\mathcal J}=\phi_\mu^{-1}({\mathbb M}_m(P))$ is a
maximal ideal in $I(\overline{\mu^0})$. Using again \cite[Lemma 3.21]{T} and the fact that $I(\overline{\mu^0})$ has local units, we have that
${\mathcal J}$ is in fact an ideal of $L(F)$. We will show that it is prime in $L(F)$. Consider then $A,B$ ideals of $L(F)$ with ${\mathcal
J}\subseteq A,B$ and $AB\subseteq {\mathcal J}$. Write $H_A=A\cap F^0$ and $H_B=B\cap F^0$. We know that $H_A,H_B\in {\mathcal H}_F$.

Suppose that $H_A,H_B\neq \emptyset$, then an application of Corollary \ref{dentrodelideal} (i) gives that $\overline{\mu^0}\subseteq H_A\cap H_B$ so
that $I(\overline{\mu^0})\subseteq I(H_A\cap H_B)=I(H_A)I(H_B)\subseteq AB\subseteq {\mathcal J} \subsetneq I(\overline{\mu^0})$, where the last
containment is proper as ${\mathcal J}$ is a maximal ideal. This is a contradiction so that this case cannot happen.

Without loss of generality we may assume that $H_A=\emptyset$, in this case we apply Corollary \ref{dentrodelideal} (ii) to obtain that $A\subseteq
I(\overline{\mu^0})$ so that ${\mathcal J}\subseteq A\subseteq I(\overline{\mu^0})$. But ${\mathcal J}$ was a maximal ideal in $I(\overline{\mu^0})$
so that ${\mathcal J}=A$, as needed.

Then since ${\mathcal J}=J/I(H)$ is prime in $L(F)\cong L(E)/I(H)$, then $J$ is certainly prime in $L(E)$.

If $J$ is a graded ideal, then ${\mathcal J}$ would be graded too. Thus we have that ${\mathcal J}=I(H_{\mathcal J})$ for $H_{\mathcal J}={\mathcal
J}\cap F^0$, and as $P\neq 0$, we have ${\mathcal J}\neq 0$ so that $H_{\mathcal J}\neq \emptyset$. Thus, an application of Corollary
\ref{dentrodelideal} (i) shows that $\overline{\mu^0}\subseteq H_{\mathcal J}$. On the other hand, since $I(H_{\mathcal J})={\mathcal J}\subseteq
I(\overline{\mu^0})$, then we have that $H_{\mathcal J}=I(H_{\mathcal J})\cap F^0\subseteq I(\overline{\mu^0})\cap F^0=\overline{\mu^0}$. That is,
$H_{\mathcal J}=\overline{\mu^0}$, and consequently ${\mathcal J}=I(\overline{\mu^0})$. This implies, via the isomorphism $\phi_\mu$, that
$P=K[x,x^{-1}]$, which contradicts the fact that $P$ is prime.

Therefore we have associated a nongraded prime ideal $J$ in $L(E)$ to any  maximal tail $M\in {\mathcal M}_\tau(E)$ and a prime ideal $P$ in
$K[x,x^{-1}]$. So that we define $\Lambda(M,P)=J$.
\end{proof}

\begin{proposition}\label{correspondencianograduado} Let $E$ be a graph. There is a bijection between $${\mathcal M}_\tau(E)\times
\Spec(K[x,x^{-1}])^*\longleftrightarrow \Spec_\tau(L(E)).$$
\end{proposition}
\begin{proof}
By following the correspondences consecutively in Propositions \ref{correspondencianograduadounlado} and \ref{correspondencianograduadoelotrolado}
one can check that $\Theta$ and $\Lambda$ are inverses one another. Concretely the equation $$\Lambda \Theta=1|_{\Spec_\tau(L(E))}$$ can be checked
with no difficulty, and the only nontrivial part of proving $$\Theta \Lambda=1|_{{\mathcal M}_\tau(E)\times \Spec(K[x,x^{-1}])^*}$$ arises when we
have $J=\Lambda(M,P)$ and we would like to establish that, in order to apply $\Theta$, we obtain $H'=H$ and therefore $M'=M$ and so on. This is so
because when defining $J$ in the $\Lambda$-process, we obtained ${\mathcal J}\cap F^0=\emptyset$ so that $J\cap E^0\subseteq H$, as $F=E/H$ and
${\mathcal J}=J/I(H)$. But the latter implies $I(H)\subseteq J$, and therefore $H=I(H)\cap E^0\subseteq J\cap E^0\subseteq H$. That is, $H'=J\cap
E^0=H$, and the rest follows trivially.
\end{proof}

Putting together Proposition \ref{correspondencianograduado} and Lemma \ref{gradedprime}, we obtain the main result of this section.

\begin{theorem}\label{correspondenciatotal} Let $E$ be a graph. There is a bijection between $${\mathcal M}(E)\cup ({\mathcal M}_\tau(E)\times
\Spec(K[x,x^{-1}])^*)\longleftrightarrow \Spec(L(E)).$$
\end{theorem}

\begin{remark} {\rm This Theorem is the algebraic version of \cite[Corollary 2.12]{HS}. Note that the role of ${\mathbb T}$ in that result
is played in Theorem \ref{correspondenciatotal} by
$\Spec(K[x,x^{-1}])^*$. This replacement agrees with the fact that
both $K[x,x^{-1}]$ and ${\mathbb T}$ are attached to the same
underlying graph in the following sense: $K[x,x^{-1}]$ is the
Leavitt path algebra of the loop graph $E$ given by

$$
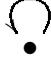
\xymatrix{{\bullet} \ar@(ur,ul) }$$ whereas the continuous functions over ${\mathbb T}$ is
precisely the graph C*-algebra of that graph, that is, $C^*(E)\cong C({\mathbb T})$.}
\end{remark}

Note that although $L(E)$ is always semiprime (see for instance \cite[Proposition 1.1]{AMMS}), is it not necessarily prime, and in fact we can prove
the following easy corollary

\begin{corollary}\label{whenisprime} Let $E$ be a graph. $L(E)$ is prime if and only if $E\in {\mathcal M}(E)$ if and only if $E$ satisfies
Condition {\rm (MT3)}.
\end{corollary}
\begin{proof}
$L(E)$ is prime if and only if $\{0\}=I(\emptyset)\in \Spec(E)$. Then by the way the correspondence in Theorem \ref{correspondenciatotal} is defined,
this occurs precisely when $E^0\setminus \emptyset=E^0\in {\mathcal M}(E)$. Then, as $\emptyset$ is always a hereditary and saturated subset of
$E^0$, Lemma \ref{maximaltailhersat} yields that $E^0$ always satisfies Conditions (MT1) and (MT2). Hence, $E^0\in {\mathcal M}(E)$ if and only if
$E^0$ satisfies Condition (MT3).
\end{proof}

\section{Primitive Leavitt path algebras}

Having completely determined the prime Leavitt path algebras, the natural next step is to be able to proceed in the same way with the primitive ones
(every primitive algebra is in particular prime, and the reverse implication holds for instance for the class of separable C*-algebras \cite{KT}, and
consequently for the class of graph C*-algebras).

In view of Corollary \ref{whenisprime}, and contrasting with the graph C*-algebra situation, next lemma shows that among the class of Leavitt path
algebras, the notions of primeness and primitivity do not coincide.

\begin{lemma}\label{Mtaunotprimitive} If $E\in {\mathcal M}_\tau(E)$, then $L(E)$ is not left (nor right) primitive.
\end{lemma}
\begin{proof}
Apply again \cite[Lemma 2.1]{HS} to find $\mu$ the only cycle without exits of $E$, and suppose that $\mu$ is based at the vertex $v$. By repeating
the arguments in \cite [Proof of Theorem 4.3]{APS} we obtain that $K[x,x^{-1}]\cong vL(E)v$, which is not a primitive ring (note that a commutative
ring is primitive if and only if it is a field). Clearly, as corners of primitive rings are primitive, then we get the result.
\end{proof}

\begin{lemma}\label{Pointlineprimitive} If $E\in {\mathcal M}(E)$ and $P_l(E)\neq \emptyset$, then $L(E)$ is left (and right) primitive.
\end{lemma}
\begin{proof}
Pick $v\in P_l(E)$ and use \cite[Theorem 2.9]{AMMS} to get that $M=L(E)v$ is a minimal left ideal of $L(E)$, or in other words, $M$ is a simple left
$L(E)$-module. Now consider $a\in L(E)$ such that $aM=aL(E)v=0$. As $v\neq 0$ and $L(E)$ is prime, we get that $a=0$, so that $M$ is a simple and
faithful left $L(E)$-module. This shows that $L(E)$ is left primitive. By proceeding dually we get that $L(E)$ is right primitive too.
\end{proof}

Recall that a ring $R$ is right primitive if and only if there exists a simple and faithful right $R$-module $M$. Given that the focus at this point
is on determining when a Leavitt path algebra $L(E)$ is (right) primitive, it is evident that a knowledge of the simple (and faithful) right
$L(E)$-modules is required. This is done in the next few results.

\begin{lemma}\label{modulosimple} If $M$ is a simple right $L(E)$-module, then $M\cong vL(E)/J$, for some
$v\in E^0$ and some right $L(E)$-module $J$, maximal (as a right $L(E)$-module) in $vL(E)$.
\end{lemma}
\begin{proof}
We know that $M\cong L(E)/I$ for some maximal right ideal $I$ of $L(E)$. Take $v\in E^0$ such that $v\not\in I$. By the maximality of $I$,
$I+vL(E)=L(E)$. So, $M\cong L(E)/I\cong (I+vL(E))/I \cong vL(E)/(I\cap vL(E))$. Observe that $J=I\cap vL(E)$ is a right $L(E)$-module, maximal in
$vL(E)$.
\end{proof}

We will denote by $\Mod$-$L$ the category of all right $L$-modules.

\begin{proposition}\label{simplesdevertices} Let $E$ be a graph. For a vertex $u\in E^0$, define the set $${\mathcal
S}_u=\{M\in \Mod\hbox{-}L\ |\ M\cong uL(E)/J, \hbox{where }J\hbox{ is a maximal right submodule of }uL(E)\}.$$ Let $u,v\in E^0$ and $\alpha$ a path
with $s(\alpha)=u$ and $r(\alpha)=v$. Then
\begin{enumerate}[{\rm (i)}]
\item ${\mathcal S}_v = {\mathcal S}_u$.
\item If $J$ is a maximal right submodule of $uL(E)$, then $uL(E)/J\cong vL(E)/\alpha^*J$.
\end{enumerate}
\end{proposition}
\begin{proof}
Denote $L(E)$ by $L$. Define the following map $$\begin{matrix} \varphi: & uL & \to & vL \cr
                                                   & x & \mapsto & \alpha^*x \end{matrix}$$
Since $\alpha^*u\alpha=v$, $\varphi$ is an epimorphism of right $L$-modules whose kernel is $(u-\alpha\alpha^*)L$. Then, $uL/\Ker(\varphi)\cong vL$
via the isomorphism $$\begin{matrix} \overline{\varphi}: & uL/\Ker(\varphi) & \to & vL \cr
                                                         & x+\Ker(\varphi) & \mapsto & \alpha^*x \end{matrix}$$

Let us see first that ${\mathcal S}_v \subseteq {\mathcal S}_u$. Let $T$ be a maximal submodule of $vL$. By using the isomorphism
$\overline{\varphi}$ we know that there exists $J$ a submodule of $uL$ such that $\Ker(\varphi)\subseteq J$ and $J/\Ker(\varphi)\cong T$. Then we
have $$vL/T\cong (uL/\Ker(\varphi))/(J/\Ker(\varphi))\cong uL/J.$$

Now we will check that ${\mathcal S}_u \subseteq {\mathcal S}_v$. Suppose first that $J+\Ker(\varphi)=uL$. Consider $$\begin{matrix} \rho: &
uL/(J\cap \Ker(\varphi)) & \to & vL \cr  & y + (J\cap \Ker(\varphi)) & \mapsto & \alpha^* y \end{matrix}$$ It is well-defined because $y\in J\cap
\Ker(\varphi)\subseteq \Ker(\varphi)$ means $y=(u-\alpha\alpha^*)y$, which implies $\alpha^*y=0$.

Clearly, it is surjective, as $ \alpha^*\alpha =v$ and $s(\alpha)=u$, and therefore $(uL/(J\cap \Ker(\varphi)))/\Ker(\rho)\cong vL$ via the
isomorphism $\overline{\rho}$ given by $(y+(J\cap \Ker(\varphi)))+\Ker(\rho) \mapsto \alpha^*y$. Apply twice the Third Isomorphism Theorem to obtain
$$uL/J\cong (uL/(J\cap \Ker(\varphi)))/(J/(J\cap \Ker(\varphi)))\cong $$ $$ \Big(\big(uL/(J\cap
\Ker(\varphi))\big)/\Ker(\rho)\Big)/\Big(\big(J/(J\cap \Ker(\varphi))+\Ker(\rho)\big)/\Ker(\rho)\Big)\cong vL/\alpha^*J$$ since $\alpha^*J$ is the
image of $\big(J/(J\cap \Ker(\varphi))+\Ker(\rho)\big)/\Ker(\rho)$ by $\overline{\rho}$.

Suppose now that $\Ker(\varphi)\subseteq J$. Then, $(uL/\Ker(\varphi))/(J/\Ker(\varphi))\cong uL/J$. As $uL/J$ is a simple module, $J/\Ker(\varphi)$
is maximal inside $uL/\Ker(\varphi)$. Using the isomorphism $\overline{\varphi}$ we have that $\overline{\varphi}(J/\Ker(\varphi))=\alpha^*J$ is a
maximal submodule of $vL$ and $uL/J\cong (uL/\Ker(\varphi))/(J/\Ker(\varphi))\cong vL/\alpha^*J$.
\end{proof}

\begin{proposition}\label{quotientedges} Let $E$ be a graph, $u$ a vertex with $|s^{-1}(u)|\geq 2$, and $uL(E)/J$ a simple
right $L(E)$-module. Then $$uL(E)/J\cong vL(E)/e^*L(E)$$ for some $e\in s^{-1}(u)$, being $v=r(e)$.
\end{proposition}
\begin{proof}
Write $L=L(E)$. Use relation (4) to write $u=ee^*+\sum_i f_if_i^*$ (note that $i\geq 1$). For every $y\in J$ we may write $y=uy=ee^*y+\sum
f_if_i^*y$, so that $J\subseteq ee^*J\oplus\big(\bigoplus f_if_i^*J\big)\subseteq uL$. By the maximality of $J$ we have two possibilities.

Case 1: $ee^*J\oplus\big(\bigoplus f_if_i^*J\big)=uL$. For any $l\in L$ write $ee^*l=ee^*a+\sum f_if_i^*b_i$ with $a,b_i\in J$. Multiply on the right
hand side by $ee^*$ to obtain $ee^*l=ee^*a\in ee^*J$. Hence, $ee^*L\subseteq ee^*J\subseteq ee^*L$, that is, $ee^*L=ee^*J$. Apply Proposition
\ref{simplesdevertices} (ii) for $\alpha=e$ to get $uL(E)/J\cong vL(E)/e^*J=vL(E)/e^*L$.

Case 2: $ee^*J\oplus\big(\bigoplus f_if_i^*J\big)=J$. In this situation $$uL/J\cong \Big(ee^*L\oplus\big(\bigoplus
f_if_i^*L\big)\Big)/\Big(ee^*J\oplus\big(\bigoplus f_if_i^*J\big)\Big)\cong ee^*L/ee^*J\oplus\big(\bigoplus f_if_i^*L/f_if_i^*J\big).$$ The
simplicity of $uL/J$ implies that every summand but one must be zero. We may suppose that $ee^*L/ee^*J=0$. Then, Proposition \ref{simplesdevertices}
(ii) applies again to have $uL(E)/J\cong vL(E)/e^*J=vL(E)/e^*L$.
\end{proof}

We now have all the ingredients in hand to prove the final result of the article.

\begin{theorem}\label{primitivos} Let $E$ be a graph. The following conditions are equivalent.
\begin{enumerate}[{\rm (i)}]
\item $L(E)$ is left primitive.
\item $L(E)$ is right primitive.
\item $E$ satisfies Conditions {\rm (L)} and {\rm (MT3)}.
\item $I\cap J\cap E^0\neq \emptyset$ for every nonzero ideals $I$ and $J$ of $L(E)$.
\end{enumerate}
\end{theorem}
\begin{proof}
(ii) $\Rightarrow$ (iii). If $L(E)$ is right primitive, then it is prime so that Proposition \ref{gradedprime} yields that $E$ satisfies Condition
(MT3). If $E$ does not satisfy Condition (L), then $E\in {\mathcal M}_\tau(E)$, and by Lemma \ref{Mtaunotprimitive}, $L(E)$ is not right primitive, a
contradiction.

(iii) $\Rightarrow$ (ii). Denote $L=L(E)$. If $P_l(E)\neq \emptyset$, we finish by Lemma \ref{Pointlineprimitive}. So, suppose $P_l(E)=\emptyset$.
Since $E$ satisfies Condition (L), there exists $u\in E^0$ with $|s^{-1}(u)|\geq 2$. Given any $v\in E^0$, by Condition (MT3) there exists $w\in E^0$
such that $u,v\geq w$. In this situation Proposition \ref{simplesdevertices} (i) gives ${\mathcal S}_u={\mathcal S}_w={\mathcal S}_v$, so that
${\mathcal S}_v={\mathcal S}_u$.

By Lemma \ref{modulosimple} every simple right module $M$ is isomorphic to $vL/J$ for some vertex $v\in E^0$ and some maximal submodule $J$ of $vL$.
Hence, Proposition \ref{quotientedges} implies that
$$\{\Ann(M)\ |\ M \hbox{ is a simple right }L\hbox{-module
}\}=$$ $$ \{\Ann(r(e)L/e^*L), \hbox{ with }e\in s^{-1}(u) \mbox{ and } r(e)L/e^*L \mbox{ simple}\} \eqno{(\dag)}$$ Clearly the second set is finite.
If all its elements are nonzero, then we can apply Proposition \ref{characterizations} (i) to get
$$\bigcap\limits_{
e\in s^{-1}(u),\ r(e)L/e^*L {\hbox{ {\tiny simple}}} } \Ann(r(e)L/e^*L) \cap E^0\neq \emptyset.$$ If we denote by $J(L)$ the Jacobson radical of $L$,
we know that $J(L)=0$ by \cite[Proposition 6.3]{AA3}. Now $(\dag)$ gives
$$J(L)=\bigcap_{M {\hbox{ {\tiny simple}}}} \Ann(M)= \bigcap\limits_{
e\in s^{-1}(u),\ r(e)L/e^*L {\hbox{ {\tiny simple}}} } \Ann(r(e)L/e^*L)\neq 0,$$ a contradiction. Thus, $\Ann (r(e)L/e^*L)=0$ for some simple
$L$-module $r(e)L/e^*L$, as desired.

(i) $\Leftrightarrow$ (iii) is proved analogously.

(iii) $\Leftrightarrow$ (iv) is Proposition \ref{characterizations} (i).
\end{proof}

\begin{remark}\label{leftright} {\rm In noncommutative Ring Theory, one-sided conditions tend not to be left-right symmetric (perhaps with the
remarkable exception of semisimplicity). However, for Leavitt path algebras, the natural phenomena seem to be the opposite: for instance, in
\cite[Theorem 3.10]{AAS1} it was shown that $L(E)$ is left noetherian if and only if it is right noetherian, and later on in \cite[Theorem 2.2]{AAPS}
the left-right symmetry was established for the artinian condition as well. Moreover, in \cite[Theorem 2.6]{AAPS} and \cite[Theorem
3.8]{AAPS}, similar situations arose for the locally artinian and locally noetherian properties.

In this sense, Theorem \ref{primitivos} adds the primitive condition to the list of left-right symmetric properties for Leavitt path algebras, and
therefore yields stronger support to the claim that $L(E)$ carries some type of extra symmetry within.}
\end{remark}

\begin{examples} {\rm In contrast with Remark \ref{leftright}, the containments
$$\{R\ |\ R \mbox{ is simple }\}\subseteq \{R\ |\ R \mbox{ is left primitive }\}\subseteq \{R\ |\ R \mbox{ is prime }\}$$ are proper for Leavitt
path algebras in the same way that they so are for general rings. To exhibit such examples, one simply uses the characterizations of prime (Corollary
\ref{whenisprime}), left primitive (Theorem \ref{primitivos}) and simple (\cite[Theorem 3.11]{AA1}) Leavitt path algebras in terms the properties of
their underlying graphs. Thus, perhaps the easiest examples can be built out of the following graphs:

$$\begin{array}{ccccc}E:\ \ & \xymatrix{{\bullet} \ar@(ur,ul) } & \phantom{aquivaunespacio} & F:\ \ \ & \xymatrix{{\bullet} \ar@(dl,ul) \ar[r] & {\bullet}}
 \end{array}$$
By using the results above, it is straightforward to check that $L(E)$ is prime but not left (nor right) primitive, whereas $L(F)$ is left (and
right) primitive but not simple. In fact $L(E)\cong K[x,x^{-1}]$ (see \cite[Examples 1.4]{AA1}), and $L(F)\cong T$ where $T$ denotes the algebraic
Toeplitz algebra (see \cite[Theorem 5.3]{S}).}
\end{examples}

\section*{Acknowledgments}
The first author gratefully acknowledges the support from the Centre de Recerca Mate\-m\`a\-tica.
Part of this work was carried out during visits of the second and third authors to the Centre de Recerca Matem\`atica, and of the first and second
authors to the Universidad de M\'alaga. They thank these host centers for their warm hospitality and support.

\end{document}